\documentclass[a4paper,12pt]{amsart}
\usepackage{amsmath,amssymb}
\usepackage{mathabx}

\def\Kb{{K^{\mathrm {sep}}}}
\def\Kbp{{K_\gp^{\mathrm {sep}}}}
\def\O{{\mathcal O}}
\def\Op{{{\mathcal O}_\gp}}
\def\Z{{\mathbb Z}}
\def\M{{\mathcal M}}
\def\Q{{\mathbb Q}}
\def\LO{{\Lambda_\O}}
\def\LZ{{\Lambda_\Z}}
\def\gp{{\mathfrak p}}
\def\gq{{\mathfrak q}}
\def\ga{{\mathfrak a}}
\def\gb{{\mathfrak b}}

\def\gf{{\mathfrak f}}
\def\gd{{\mathfrak d}}
\def\gf{{\mathfrak f}}
\def\gr{{\mathfrak r}}
\def\LOp{{\Lambda_{\O_\gp}}}
\def\DR{{\mathrm{DR}}}
\def\Map{{\mathrm{Map}}}
\def\alg{{\mbox{-}\mathrm{alg}}}
\def\Hom{{\mathrm{Hom}}}
\def\End{{\mathrm{End}}}
\def\lcm{{\mathrm{lcm}}}
\def\fin{{\mathrm{fin}}}
\def\unr{{\mathrm{unr}}}
\def\ord{{\mathrm{ord}}}
\def\Cl{{\mathrm{Cl}}}
\def\G{{G_K}}
\def\I{{I_K}}
\def\Gb{{\bar{G}}}
\def\Gp{{G_\gp}}
\def\Ip{{I_\gp}}

\def\longrightisomap{{\buildrel \sim\over\longrightarrow}}
\newcommand{\longmap}{{\,\longrightarrow\,}}

\newtheorem{thm}[subsection]{Theorem}
\newtheorem{Proposition}[subsection]{Proposition}

\begin{document}
\title{Lambda actions of rings of integers}
\author[J.~Borger, B.~de~Smit]{James Borger, Bart de Smit}
\email{james.borger@anu.edu.au, desmit@math.leidenuniv.nl}
\thanks{This preprint is a preliminary version dating from 2006. We are making it available in this
form because some people would like to cite it now. The final version should be available before long.}

\begin{abstract}
Let $\O$ be the ring of integers of a number field $K$.
For an $\O$-algebra $R$ which is torsion free as an $\O$-module
we define what we mean by a $\Lambda_\O$-ring structure on $R$.
We can determine whether a finite \'etale $K$-algebra $E$ with
$\Lambda_\O$-ring structure has an integral model in terms of a
Deligne-Ribet monoid of $K$. This a commutative monoid whose
invertible elements form a ray class group.
\end{abstract}

\maketitle

\section{Introduction} Let $\O$ be a Dedekind domain with quotient
field $K$. Denote the set of maximal ideals of $\O$ by $\M$.  We
assume that $k(\gp)=\O/\gp$ is a finite field for each $\gp\in\M$.

Let $E$ be a torsion-free commutative $\O$-algebra. Then for
each $\gp\in\M$ the algebra $E/\gp E= E \otimes_\O k(\gp)$ over
$k(\gp)$ has a natural $k(\gp)$-algebra endomorphism
$F_\gp\colon x\mapsto x^{\#k(\gp)}$, which is called the Frobenius
endomorphism.  By a Frobenius lift of $E$ at $\gp$ we mean an
$\O$-algebra endomorphism $\psi_\gp$ such that $\psi_\gp\otimes
k(\gp)=F_\gp$.  We define a $\LO$-structure on $E$ to be a map
$\M\to\End_{\O\alg}(E)$, denoted $\gp\mapsto\psi_\gp$, such that 

\begin{enumerate}
	\item 
$\psi_\gp$ is a Frobenius lift at $\gp$ for each $\gp\in M$.
	\item 
$\psi_\gp \psi_\gq =\psi_\gq \psi_\gp$ for all $\gp$, $\gq\in\M$.
\end{enumerate}
By a $\LO$-ring we mean a torsion-free $\O$-algebra with $\LO$-structure.

If $\O$ is local then the commutation condition (2) is vacuous.  For
all $\gp\in\M$ for which $E\otimes k(\gp)=0$ the lifting condition (1)
is vacuous.  In particular, if $E$ is an algebra over $K$, then any
commuting collection of $K$-automorphisms of $E$ indexed by the
maximal ideals of $\O$ is a $\LO$-structure on $E$.

A $\LZ$-structure on a ring without $\Z$-torsion is the same as a
$\lambda$-ring structure \cite{Wilkerson}. For instance, for any abelian group $A$ we
have a natural $\LZ$-stucture on the group ring $\Z[A]$ given by
$\psi_p(a)=a^p$ for $a\in A$ and $p$ a prime number. 

If $\O$ is the ring of integers of a number field $K$, and 
$E$ is the ring of integers of a subfield $L$ of the strict Hilbert
class field of $K$, then $E$ has a unique $\LO$-structure: 
$\psi_\gp$ is the Artin symbol of $\gp$ in the field
extension $K\subset L$.


In an earlier paper \cite{Borger-deSmit:Integral-models}, we showed that a $\LZ$-ring that is reduced
and finite flat over $\Z$ is a $\LZ$-subring of $\Z[C]^n$ for some
finite cyclic group $C$ and positive integer $n$. The proof uses the
explicit description of ray class fields over $\Q$ as cyclotomic
fields. Over a number field class field theory is less explicit, and
the generalizations we present in the present paper are by consequence
less explicit.  However, we can still give a very similar criterion
for a $\LO$-structure on a finite \'etale $K$-algebra $E$ to come from
an $\LO$-subring which is finite flat as an $\O$-module see Theorem
\ref{global-thm} below.
Such a $\LO$-subring is called an integral $\LO$-model of the
$\LO$-ring~$E$.

Let $I(\O)$ be the monoid of non-zero ideals of $\O$, with ideal
multiplication as the monoid operation.  It is the free commutative
monoid on $\M$.

Let $\Kb$ be a separable closure of $K$, and let $\G$ be the
Galois group of $\Kb$ over $K$. It is a profite group. By a $\G$-set
$X$ we mean a finite discrete set with a continuous $\G$-action.
By Grothendieck's formulation of Galois theory, a finite
\'etale $K$-algebra $E$ is determined by the $\G$-set $S$ consisting
of all $K$-algebra homomorphisms $E\to \Kb$.
Giving a $\LO$-structure on $E$ then translates to giving a monoid map
$I(\O)\to\Map_\G(S,S)$.  By giving $I(\O)$ the discrete topology, we
see that the category of $\LO$-rings whose underlying $\O$-algebra is
a finite \'etale $K$-algebra, is anti-equivalent to the category of
finite discrete sets with a continuous action of the monoid
$I(\O)\times \G$.

Let us first suppose that $\O$ is complete discrete valuation ring with
maximal ideal $\gp$. Then $I(\O)$ is isomorphic as a monoid to the
monoid of non-negative integers with addition.  Let $\I\subset\G$ be
the inertia subgroup.  Then $\I$ is normal in $\G$ and $\G/\I$ is the
absolute Galois group of $k(\gp)$, which contains the Frobenius
element $F\in \G/\I$ given by $x\mapsto x^{\#k(\gp)}$. Thus, $F$ acts
on any $\G$-set on which $\I$ acts trivially.

\begin{thm}
\label{local-thm}
Suppose $\O$ is complete discrete valuation ring with maximal ideal $\gp$.
Let $E$ be a finite \'etale $K$-algebra with $\LO$-structure,
and let $S$ be the set of $K$-algebra maps from $E$ to $\Kb$.
Then $K$ has an integral $\LO$-model if and only if the
action of $I(\O)\times \G$ on $S$ satisfies the two conditions
\begin{enumerate}
\item the group $\I$ acts trivially on $S_{\unr}=\bigcap_{\ga\in I(\O)}\ga S$;
\item $\gp\in I(\O)$ and $F\in \G/\I$ act in the same way on $S_{\unr}$.
\end{enumerate}
\end{thm}

\medskip\noindent
See Section 2 for the proof.

Next, let us assume that $\O$ is the ring of integers in a number
field. 

In order to phrase our global result we first recall the definition
of the Deligne-Ribet
monoid. A cycle of $K$ is a formal product $\gf=\prod_\gp\gp^{n_\gp}$,
where the product ranges over all primes of $K$, both finite and
infinite, all $n_\gp$ are non-negative integers, only finitely many of
which are non-zero, and we have $n_\gp\in\{0,1\}$ for real primes
$\gp$, and $n_\gp=0$ for complex primes $\gp$. 
The finite part of $\gf$ is
$\gf^{\fin}=\prod_{\gp<\infty}\gp^{n_\gp}$, which can be viewed as an
element of $I(\O)$.  We write $\ord_\gp(\gf)=n_\gp$.

For a cycle $\gf$ we say that two non-zero $\O$-ideals $\ga$ and
$\gb$ are $\gf$-equivalent if $x\ga=\gb$ for some $x\in K^*$
with $x>0$ at all real places $\gp$ with $\ord_\gp(\gf)>0$, and
$\ord_\gp(x-1) +\ord_\gp(\ga)\ge \ord_\gp(\gf)$ at all finite places $\gp$.
One can check that this is an equivalence relation, and that the
multiplication of ideals is well-defined on the quotient set.
Thus, the quotient set is a monoid, the Deligne-Ribet-monoid,
and we denote it by $\DR(\gf)$.

It is not hard to see that the ray class group $\Cl(\gf)$ is
the group of invertible elements of $\DR(\gf)$. 
Also, $\DR(1)$ is a group: it is the class group of $\O$.
More generally, for each ideal $\gd$ dividing 
$\gf^{\fin}$ we can consider the
map $i_\gd\colon\; \Cl(\gf/\gd)\to \DR(\gf)$ that sends the class of an ideal $\ga$
to the class of $\ga\cdot \gd$. These maps give rise to
a bijection
\[
i={\textstyle \coprod} i_\gd\colon\;
\coprod_{\gd\mid\gf^\fin} \Cl(\gf/\gd) \longrightisomap \DR(\gf).
\]

\begin{thm}
\label{global-thm}
Suppose $\O$ is the ring of integers of a number field $K$.
Let $E$ be a finite \'etale $K$-algebra with $\LO$-structure,
and let $S$ be the set of $K$-algebra maps from $E$ to $\Kb$.
Then $K$ has an integral $\LO$-model if and only
if there is a cycle $\gf$ of $K$ so that the action of
$G_{\Kb}\times I(\O)$ on $S$ factors (necessarily uniquely)
through the map
\[
\G\times I(\O) \longmap \DR(\gf),
\]
which is the product of the Artin symbol $\G\to \Cl(\gf)\subset
\DR(\gf)$ on the first coordinate, and the quotient map
$I(\O)\to\DR(\gf)$ on the second.  
\end{thm}

It follows that the category of such $\Lambda$-rings is anti-equivalent to
the category of finite discrete sets with a continuous action by the
profinite monoid $
\displaystyle {\lim_{\longleftarrow}}{}\;
\DR(\gf)$, where
the limit is taken over all cycles $\gf$ with respect to the canonical
maps $\DR(\gf)\to \DR(\gf')$ when $\gf'\mid \gf$.
When $K=\Q$ this limit is the multiplicative monoid of profinite integers.

\section{The local case}

Suppose that $\O$ is a complete discrete valuation ring with maximal
ideal $\gp$. We write $k=k(\gp)$.  Let $A$ be a reduced finite flat
$\O$-algebra. 

Let us suppose first that $A$ is unramified over $\O$, i.e., that
$k\otimes_\O A$ is \'etale over $k$. Then $k\otimes_\O A$ is a product
of finite fields.  Since $A$ is complete in its $\gp$-adic topology,
idempotents of $A/\gp A$ lift to $A$, so that $A$ is a finite product
of rings of integers in finite unramified extensions of $K$. 
Write 
\[
S=\Hom_{\O\alg}(A,\Kb)=\Hom_{K\alg}(A\otimes_\O K, \Kb).
\]
Then the inertia group $\I\subset \G$ acts trivially on $S$.  Every
finite unramified field extension $L$ of $K$ is Galois with an abelian
Galois group, and its rings of integers has a unique Frobenius lift,
which is othen called the Frobenius element of the Galois group of $L$
over~$K$.  It follows that when $A$ is unramified over $\O$, it has a
unique $\LO$-structure. This is summarized in the next Proposition.

\begin{Proposition}
\label{unramified}
Suppose that $\O$ is a complete discrete valuation ring, and that $A$
is an unramified finite flat reduced $\O$-algebra.
Then $A$ has a unique
$\LO$-structure, and the induced action of $I(\O)\times \G$ on
$S=\Hom_{K\alg}(A\otimes_\O K, \Kb)$ has the
property that the intertia group $\I$ acts trivially and that
$\gp\in\I(\O)$ acts in the same way on $S$ as $F\in\G/\I$.
\end{Proposition}

\begin{proof}[Proof of Theorem \ref{local-thm}]
Put $S_0=S_\unr$ and for $i=1,2, \ldots 1$ let $S_i$ be the set of all
$s\in S$ with $s\not\in S_{i-1}$ and $\gp s \in S_{i-1}$.
Suppose that $S_n\ne\varnothing$, and $S_{n+1}=\varnothing$.
Let $E_i=\Map_{\G}(S_i, \Kb)$ be the corresponding finite \'etale
$K$-algebra for each $i$. Then multiplication by $\gp$
gives rise to $K$-algebra homomorphisms $f_i\colon\;E_{i-1}\to E_{i}$ for
$i=1, 2, \ldots$, and $f_0\colon\; E_0\to E_0$.
\[
f_0 \lefttorightarrow
E_0 {\buildrel f_1 \over\longrightarrow}
E_1 {\buildrel f_2 \over\longrightarrow}
\cdots
{\buildrel f_{n-1} \over\longrightarrow}
E_n
\]
Since $S=S_0\coprod S_1 \coprod \cdots\coprod S_n$ is a
decomposition of $G$ as a $\G$-set, and we have a corresponding
product decomposition of the finite \'etale $K$-algebras
$E=E_0\times \times \cdots\times E_n$. 
In terms of this decomposition $\psi_p$ is given by
\[
\psi_\gp(e_0,e_1, \ldots, e_n)=(f_0(e_0),f_1(e_0),\ldots, f_{n-1}(e_{n-1})).
\]

Since $S_0$ is closed under multiplication by $\gp$, the quotient ring
$E_0$ of $E$ is a quotient $\LO$-ring of $E$, with Frobenius lift
$f_0$ at $\gp$.
We will show that the $\LO$-ring surjection $E\to E_0$ splits.

Note that now $\gp^kS=S_0$ for sufficiently large
$k$, so $\gp$ act as a bijection on $S_0$. Thus, $f_0$ is an
automorphism of $E_0$.
For $s\in S_i$ we have $\gp^is\in S_0$ and $\gp$ acts invertibly on
$S_0$, so we can define a map 
$S\longmap S_0$ by sending $s\in S_i$ to $\gp^{-i}(\gp^is)$.
This map commutes with the $I(\O)\times \G$-action, and it splits
the inclusion $S_0\to S$. Thus, $E_0$ is not only a quotient
 $\LO$-ring of $E$, but also a sub-$\LO$-ring:
\begin{align*}
i\colon\; E_0  \longmap  & E \\
e_0  \longmapsto & (e_0,f_1f_0^{_1}e_0, f_2f_1f_0^{-2} e_0, \ldots,
f_{n-1} \cdots f_1 f_0^{-n+1} e_0).\\
\end{align*}

Now suppose that the $\LO$-ring $E$ has an integral model, i.e.,
that $E$ has an $\O$-sub algebra $A$ which statisfies
\begin{enumerate}
	\item $A$ is finite flat over $\O$;
	\item $\psi_p(A)\subset A$;
	\item $\psi_p\otimes_\O k$ is the Frobenius $x\mapsto x^{\#k}$
		on $A\otimes k$.
\end{enumerate}
The image $A_0$ of $A$ in the quotient ring $E_0$ of $E$ is a
sub-$\LO$-ring of $E_0$ which is reduced and finitely generated as an
$\O$-module and $\O$-torsion free. Thus, $E_0$ has an integral
$\LO$-model. Since $f_0$ is an automorphism of $E_0$ the rings $A_0$
and its subring $f(A_0)$ have the same discriminant. Thus,
$f_0(A_0)=A_0$ and $f_0$ is an automorphism of $A_0$.  This implies
that the map $x\mapsto x^{\#k}$ on $A_0\otimes_\O k$ is an
automorphism, so that $A_0$ is unramified over $\O$.  Conditions (1)
and (2) of Theorem \ref{local-thm} now follow by Proposition
\ref{unramified}.

For the converse, suppose that conditions (1) and (2) hold.
We will produce an integral $\LO$-model of $E=E_0\times \cdots\times E_n$.
Let $R_i$ be the inegral closure of $\O$ in $E_i$.
Since $\I$ acts trivially on $S_0$ the ring $R_0$ has a
unique $\LO$-structure by Proposition \ref{unramified}. 
Now suyppose that
\[
A = i (R_0) \oplus  (0\times \ga_1\times \cdots\times\ga_n)
\]
with $\ga_i$ an ideal in $R_i$. Then the condition
$\psi_p(a)-a^{\#k(\gp)}\in \gp A$
for all $a\in A$ is equivalent to $\ga_i^{\#k(\gp)}\subset\gp\ga_i$ and
$f_i(\ga_{i-1})\subset \gp\ga_{i}$. This holds, for instance
if $\ga_i=\gp^iR_i$, in which case $A$ is an integral $\LO$-model of $E$.
\end{proof}


The integral model that is supplied by the proof is not always
optimal. For instance, for the $\LZ$-ring $\Z[C_4]$ we get a strict
subring. However for the $\LZ$-ring $\Z[V_4]$ the proof provides a
$\LZ$-subring of $\Q[V_4]$ which is strictly larger than $\Z[V_4]$.

\section{Global arguments}
Now assume that $K$ is a global field with ring of integers $\O$.
Let $E$ be a finite \'etale $K$-algebra with a $\LO$-sturcture.
Writing $S=\Hom_K(E,\Kb)$ we thus get an action of
$I(\O)\times \G$ on $S$.

For each maximal ideal $\gp$ of $\O$ we consider the completion
$\O_\gp$, and its quotient field $K_\gp$. Then we obtain an
$\LOp$-structure on the finite \'etale $K_\gp$-algebra $E_\gp=
E\otimes_K K_\gp$. If $A$ is an integral $\LO$-model of $E$, then 
$A\otimes _\O\O_\gp$ is an integral $\LOp$-model of $E_\gp$.

Fixing an embedding $\Kb\to \Kbp$ for each $\gp$ we can view $\Gp$ as
a subgroup of $\G$. The finite \'etale $K_\gp$-algebra $E_\gp$ then
corresponds to the $\Gp$-set that one gets by restricting the action
of $\G$ on $S$ to $\Gp$.

Let us assume that an integral $\LO$-model $A$ of $E$
exists. Let $\Gb$ be the image of $\G$ in $\Map(S,S)$.
Chebotarev's theorem now implies the following: for each
$g\in\Gb$ there is a maximal ideal $\gp=\gp_g$ of $\O$ so that
\begin{enumerate}
	\item the image of $\Ip$ is trivial in $\Gb$;
	\item the image of $F_\gp\in \Gp/\Ip$ in $\Gb$ is $g$;
	\item $A$ is unramified at $\gp$.
\end{enumerate}
By Proposition \ref{unramified}, the action of $g$ on $S$ is the same as the 
action of $\gp_g$ on $S$. Since the $\gp_g$ commute with eachother,
it follows that $\Gb$ is abelian.

It remains top show that the $I(\O)\times \G$-action on $S$ factors through
the Deligne-Ribet monoid of some cycle $\gf$. 

By class field theory, any continuous action of $\G$ on a finite
discrete set $T$, whose image is abelian, factors, by the Artin map,
through the ray class group $\Cl(c(T))$ for a minimal cycle $c(T)$
of $K$, which we call the conductor of $T$. 

Define $\gr\in I(\O)$ by setting
\[
\ord_\gp(\gr)=\inf\{i\ge 0\colon\; \gp^{i+1}S=\gp^i S\}
\]
for all maximal ideals $\gp$ of $\O$. This is well defined because
$\gp S=S$ whenever $\gp$ is unramified in $A$ by Proposition
{unramified}.

We now define the cycle $\gf$ by
\[
\gf={\displaystyle\lcm}_{\gd\mid \gr} \gd \cdot c(\gd S).
\]
Note first that $ c(S)\mid \gf$, so the $\G$-action on $S$ factors
through the Artin map $\G\to \Cl(\gf)$.

Next, we claim that for $\ga\in I(\O)$ coprime to $\gf$ the action of
$\ga$ on $S$ is equal that of its class $[\ga]$ in $\Cl(\gf)$.  It
suffices to prove this for $\ga=\gp$ prime. Then one notes that
$\gr\mid \gf$ so $\gp\nmid \gr$ so $\gp$ acts as a bijection on $S$.
By our local result, $A$ is unramified at $\gp$ and $\gp$ acts as
$F_\gp\in\Gp/\Ip$ on $S$. By the defnition of the Artin symbol, the
action of $[\gp]\in\Cl(\gf)$ is the same. This shows the claim.

Now suppose that $\gd\in I(\O)$ with $\gd\mid \gf$.
Let us write $I_\gd$ for the submonoid of $I(\O)$ consisting
of all $\ga\in I(\O)$ that are coprime to $\gf/\gd$.
We now claim that $I_\gd$ acts by bijections on $\gd S$,
by the definition of $\gf$ is quotient of $\Cl(\gf/\gd)$.
and that the action factors through $\Cl(\gf/\gd)$.
To see this, let $\ga=\gcd(\gd,\gr)$ and write $\gd=\ga\gb$.
By definition of $\gr$ all prime divisors of $\gb$ act bijectively
on $\gd S$, so $c(\ga S)=c(\gd S)$ is coprime to $\gb$.
By definition of $\gf$ we have $c(\ga S) \mid \gf/ga$, and it follows
that $c(\ga S)\mid \gf/\gd$. Thus, the $\G$-action on $\gd S$ 
factors through $\Cl(\gf/\gd)$ and the claim holds.

Since multiplication by any divisor $\gd\in I(\O)$
of $\gf$ gives a bijection
$\Cl(\gf/\gd)\to [\gd]\DR(\gf)^*$ this shows that
the action of $\{\ga\in  I(\O)\colon \gcd(\ga,\gf)=\gd\}$
on $S$ factors through $ [\gd]\DR(\gf)^*$.
Taking the union over all $\gd$ we see that the
$I(\O)$-actions factors through $\DR(\gf)$.

For the converse, assume that the $\G\times I(\O)$-action on $S$
factors through $\DR(\gf)$ for some cycle $\gf$. 

We first show the existence, for each $\gp\in \M$ 
of an integral $\LOp$-model for the $\LOp$-ring $E\otimes_K K_\gp$. 
For $\gp\nmid \gf$ this follows from the definition of the
Artin map and Proposition \ref{unramified}.
So assume $\gp\mid \gf$, and write $\gf=\gp^n\gf'$ with
$\gp\nmid\gf'$. Then $[\gp^k] \in [\gp^n]\Cl(\gf')\subset \DR(\gf)$
for all $k\ge n$. This implies that the action of $\gp$ on
$\bigcap_i\gp^i S=\gp^nS$ is given by the
Artin symbol of $[\gp]\in \Cl(\gf')$, which by Theorem \ref{local-thm}
guarantees existence of an integral $\LOp$-model.

Now let $R$ be the integral closure of $\O$ in $E$. Then
$R$ is finite flat over $\O$ and $E=R\otimes _\O K$.
For all $\gp\nmid \gf$ we are in the unramified case,
and our integral $\LOp$-module is equal to $R\otimes_\O\Op$.
It follows that the intersection $A$ over all $\gp$ of our integral
$\LOp$-module gives a sub-$\O$-algebra of $R$, which is
of finite index, and which is closed under all $\psi_\gp$.
Also, each $\psi_\gp$ are Frobenius lifts, since $A\otimes_\O \Op$
is a $\LOp$-ring. This proves Theorem \ref{global-thm}

\bibliography{references}
\bibliographystyle{plain}

\end{document}